\documentclass[11pt]{amsart}

\usepackage{geometry}
\usepackage[all]{xy}
\usepackage{epstopdf}
\usepackage{amssymb,amsmath,amsfonts,amsthm,graphicx,enumerate,amscd}
\usepackage{mathrsfs}

\theoremstyle{plain}
\newtheorem{theorem}{Theorem}[section]

\newtheorem{corollary}[theorem]{Corollary}
\newtheorem{lemma}[theorem]{Lemma}
\newtheorem{proposition}[theorem]{Proposition}

\theoremstyle{definition}

\newtheorem{remark}[theorem]{Remark}

\newtheorem{definition}[theorem]{Definition}

\title{$KR$-theory of compact Lie groups with group anti-involutions}
\author{Chi-Kwong Fok}
\date{August 16, 2015}

\begin{document}
\maketitle
\begin{abstract}
	Let $G$ be a compact, connected, and simply-connected Lie group, equipped with an anti-involution $a_G$ which is the composition of a Lie group involutive automorphism $\sigma_G$ and the group inversion. We view $(G, a_G)$ as a Real $(G, \sigma_G)$-space via the conjugation action. In this note, we exploit the notion of Real equivariant formality discussed in \cite{Fo} to compute the ring structure of the equivariant $KR$-theory of $G$. In particular, we show that when $G$ does not have Real representations of complex type, the equivariant $KR$-theory is the ring of Grothendieck differentials of the coefficient ring of equivariant $KR$-theory over the coefficient ring of ordinary $KR$-theory, thereby generalizing a result of Brylinski-Zhang's (\cite{BZ}) for the complex $K$-theory case. 
\end{abstract}

\section{Introduction}
Let $G$ be a compact, connected and simply-connected Lie group, viewed as a $G$-space via the conjugation action. According to the main result of \cite{BZ}, the equivariant $K$-theory ring $K_G^*(G)$ is isomorphic to $\Omega_{R(G)/\mathbb{Z}}$, the ring of Grothendieck differentials of the complex representation ring of $G$ over the integers (in fact, Brylinski-Zhang proved that this is true for $\pi_1(G)$ being torsion-free). Assuming further that $G$ is equipped with an involutive automorphism $\sigma_G$, the author gave in \cite{Fo} an explicit description of the ring structure of the equivariant $KR$-theory (cf. \cite{At2}, \cite{At3} and \cite{AS} for definition of $KR$-theory) $KR^*_{G}(G)$ by drawing on Brylinski-Zhang's result, Seymour's result on the module structure of $KR^*(G)$ (cf. \cite{Se}) and the notion of Real equivariant formality. $KR_{G}^*(G)$ in general has far more complicated ring structure and, among other things, is not a ring of Grothendieck differentials, as one would expect from Brylinski-Zhang's theorem. This is because in general the algebra generators of the equivariant $KR$-theory ring do not simply square to 0. 

In this note, we equip $G$ instead with an anti-involution $a_G:=\sigma_G\circ \text{inv}$. Denoting the $(G, \sigma_G)$-space $(G, a_G)$ by $G^-$ for brevity, we compute the ring structure of $KR^*_{G}(G^-)$ following the idea of \cite{Fo}. We find that there exists a derivation of the graded ring $KR^0_{G}(\text{pt})\oplus KR^{-4}_{G}(\text{pt})$ taking values in $KR^1_{G}(G^-)\oplus KR_{G}^{-3}(G^-)$ (cf. Proposition \ref{antider}) and that any element in the image of the derivation squares to 0 (see Propositions \ref{cpxvecbdleinv}, \ref{antider} and \ref{sqzeroantibz}(\ref{sqzero}), and compare with \cite[Theorem 4.30, Proposition 4.31]{Fo}). In particular,
\begin{theorem}
	If $G$ does not have any Real representation of complex type with respect to $\sigma_G$, then the derivation in Proposition \ref{antider} induces the following ring isomorphism
	\[KR^*_{G}(G^-)\cong\Omega_{KR^*_{G}(\text{pt})/KR^*(\text{pt})}\]
\end{theorem}
Hence an anti-involution is the `right' involution needed to generalize Brylinski-Zhang's result in the context of $KR$-theory. As a by-product, we also obtain the following
\begin{corollary}
	If $G$ is a compact connected Real Lie group (not necessarily simply-connected) and $X$ a compact Real $G$-space, then for any $x$ in $KR^1_G(X)$ or $KR_G^{-3}(X)$, $x^2=0$. 
\end{corollary}
Note that graded commutativity only implies that $x^2$ is 2-torsion. 

Throughout this note, $G$ is a compact, connected and simply-connected Lie group unless otherwise specified. We sometimes omit the notation for the involution when it is clear from the context that a Real structure is implicitly assumed.

\section{Background}
In this section, we recall some relevant definitions and results from \cite{BZ} and \cite{Fo} needed in this note. We refer the reader to \cite{At2}, \cite{At3} and \cite{AS} for the basic definition of (equivariant) $KR$-theory, which we shall omit here.
\begin{definition} Let $G$ be a compact Lie group equipped with an involutive automorphism $\sigma_G$, i.e. a Real compact Lie group, and $X$ a finite $CW$-complex equipped with an involution. 
	\begin{enumerate}
		\item(cf. \cite[Proposition 2.29]{Fo}) Let $c: KR_G^*(X)\to K_G^*(X)$ be the \emph{complexification map} which forgets the Real structure of Real vector bundles, and $r: K_G^*(X)\to KR_G^*(X)$ be the \emph{realification map} defined by
		\[[E]\mapsto [E\oplus\sigma_X^*\sigma_G^*\overline{E}]\]
		where $\sigma_G^*$ means twisting the original $G$-action on $E$ by $\sigma_G$.
		\item(cf. \cite[Definitions 2.1 and 2.5]{Fo}) Let $\delta: R(G)\to K^{-1}(G)$ be the derivation of $R(G)$ taking values in the $R(G)$-module $K^{-1}(G)$ (the module structure is realized by the augmentation homomorphism), where $\delta(\rho)$ is represented by the complex of vector bundles
		\begin{align*}
			0\to G\times\mathbb{R}\times V &\to G\times\mathbb{R}\times V\to 0,\\
			(g, t, v) &\mapsto\  (g, t, -t\rho(g)v) \text{ if } t\geq 0,\\
			(g, t, v) &\mapsto\  (g, t, tv) \text{ if }t\leq 0.
		\end{align*}
		We define $\delta_G: R(G)\to K^{-1}_G(G)$ similarly. $\delta_G(\rho)$ is represented by the same complex of vector bundles where $G$ acts on $G\times\mathbb{R}\times V$ by $g_0\cdot(g_1, t, v)=(g_0g_1g_0^{-1}, t, \rho_V(g_0)v)$.		
		\item Let $\sigma_n$ be the class of the standard representation of $U(n)$ in $R(U(n))$. Let $T$ be the standard maximal torus of $U(n)$. Let $\sigma_\mathbb{R}$ be the complex conjugation on $U(n)$, $T$, $U(n)/T$ or $U(\infty)$. Let $\sigma_\mathbb{H}$ be the symplectic type involution on $U(2m)$ (given by $g\mapsto J_m\overline{g}J_m^{-1}$), $U(2\infty)$ or the involution $gT\mapsto J_m\overline{g}T$ on $U(2m)/T$. Let $a_\mathbb{R}$ and $a_\mathbb{H}$ be the corresponding anti-involutions on the unitary groups. 
		\item(cf. \cite[Definitions 2.9 and 2.20]{Fo}) A Real representation $V$ of $G$ is a finite-dimensional complex representation of $G$ equipped with an anti-linear involution $\sigma_V$ such that $\sigma_V(g\cdot v)=\sigma_G(g)\cdot\sigma_V(v)$. Similarly a Quaternionic representation is one equipped with an anti-linear endomorphism $J_V$ such that $J_V^2=-\text{Id}_V$ and $J_V(g\cdot v)=\sigma_G(g)\cdot J_V(v)$. For $\mathbb{F}=\mathbb{R}$ or $\mathbb{H}$, a morphism between $V$ and $W\in \mathcal{R}ep_\mathbb{F}(G)$ is a linear transformation from $V$ to $W$ which commutes with $G$ and respects both $\sigma_V$ and $\sigma_W$. Let $\mathcal{R}ep_\mathbb{R}(G)$ (resp. $\mathcal{R}ep_\mathbb{H}(G)$) be the category of Real (resp. Quaternionic) representations of $G$. The Real (resp. Quaternionic) representation group of $G$, denoted by $RR(G)$ (resp. $RH(G)$) is the Grothendieck group of $\mathcal{R}ep_\mathbb{R}(G)$ (resp. $\mathcal{R}ep_\mathbb{H}(G)$). 
		\item(cf. \cite[Definitions 2.11 and 2.20]{Fo}) Let $V$ be an irreducible Real (resp. Quaternionic) representation of $G$. Its \emph{commuting field} is defined to be $\text{Hom}_G(V, V)^{\sigma_V}$, which is isomorphic to either $\mathbb{R}$, $\mathbb{C}$ or $\mathbb{H}$. Let $RR(G, \mathbb{F})$ (resp. $RH(G, \mathbb{F})$) be the abelian group generated by the isomorphism classes of irreducible Real (resp. Quaternionic) representations with $\mathbb{F}$ as the commuting field.
		\item(cf. \cite[Definition 2.19]{Fo}) Let $R(G, \mathbb{C})$ be the abelian group generated by the isomorphism classes of those irreducible complex representations $V$ satisfying $V\ncong\sigma_G^*\overline{V}$. 
	\end{enumerate}
\end{definition}

\begin{theorem}\label{eqderivation}
	Let $G$ be a compact connected Lie group with torsion-free fundamental group. 
	\begin{enumerate}
		\item \cite{Ho} $K^*(G)$ is a free $\mathbb{Z}$-module. If $P$ is the image of the map $\delta: R(G)\to K^{-1}(G)$, then 
		\[K^*(G)\cong \bigwedge\nolimits_\mathbb{Z}^* P\]
		\item \cite{BZ} Let $\varphi: \Omega_{R(G)/\mathbb{Z}}\to K_G^*(G)$ be the $R(G)$-algebra homomorphism defined by the following
	\begin{enumerate}
		\item $\varphi(\rho_V):=[G\times V]\in K_G^*(G)$, where~$G$ acts on $G\times V$ by $g_0\cdot(g_1, v)=(g_0g_1g_0^{-1}, \rho_V(g_0)v)$, and
		\item $\varphi(d\rho_V)=\delta_G(\rho)$. 
	\end{enumerate}
	Then $\varphi$ is an $R(G)$-algebra isomorphism.
	\end{enumerate}
\end{theorem}

\begin{definition}\label{equivformal}
	\begin{enumerate}
		\item (\cite[Definition 4.1]{HL}) A $G$-space $X$ is \emph{weakly equivariantly formal} if the map $K^*_G(X)\otimes_{R(G)}\mathbb{Z}\to K^*(X)$ induced by the forgetful map is a ring isomorphism, where $\mathbb{Z}$ is viewed as an $R(G)$-module through the augmentation homomorphism. 
		\item \label{realequivformaldef}(\cite[Definition 4.2]{Fo}) A Real $G$-space $X$ is a \emph{Real equivariantly formal} space if 
		\begin{enumerate}
			\item $G$ is a Real compact Lie group,
			\item $X$ is a weakly equivariantly formal $G$-space, and
			\item the forgetful map $KR_G^*(X)\to KR^*(X)$ admits a section $s_R: KR^*(X)\to KR_G^*(X)$ which is a $KR^*(\text{pt})$-module homomorphism.
		\end{enumerate}
		\item (\cite[Definition 4.4]{Fo})For a section $s: K^*(X)\to K_G^*(X)$ (resp. $s_R: KR^*(X)\to KR^*_G(X)$) and $a\in K^*(X)$ (resp. $a\in KR^*(X)$), we call $s(a)$ (resp. $s_R(a)$) a \emph{(Real) equivariant lift} of $a$, with respect to $s$ (resp. $s_R$). 
	\end{enumerate}
\end{definition}
\begin{remark}\label{gequivformal}
	By Theorem \ref{eqderivation} and noting that the forgetful map takes $\delta_G(\rho)\in K_G^{-1}(G)$ to $\delta(\rho)\in K^{-1}(G)$, we have that $G$ is weakly equivariantly formal if it is a compact connected Lie group with torsion-free fundamental group. 
\end{remark}
The following result is a structure theorem for the equivariant $KR$-theory of Real equivariantly formal spaces.
\begin{theorem}\label{equivstrthm}\textnormal(\cite[Theorem 4.5]{Fo}\textnormal) Let $X$ be a~Real equivariantly formal space. For any element $a\in K^*(X)$ $($resp. $a\in KR^*(X))$, let $a_G\in K_G^*(X)$ $($resp. $a_G\in KR_G^*(X))$ be a~$($Real$)$ equivariant lift of $a$ with respect to a group homomorphic section $s$ $($resp.~$s_R$ which is a~$KR^*(\text{pt})$-module homomorphism$)$. Then the map
\begin{align*}
 f:  \ (RR(G, \mathbb{R})\oplus RH(G, \mathbb{R}))\otimes KR^*(X)\oplus r(R(G, \mathbb{C})\otimes K^*(X)) & \to  KR_G^*(X),
\\
\rho_1\otimes a_1\oplus r(\rho_2\otimes a_2) & \mapsto  \rho_1\cdot (a_1)_{G}\oplus r(\rho_2\cdot (a_2)_G).
\end{align*}
is a group isomorphism. In particular, if $R(G, \mathbb{C})=0$, then $f$ is a~$KR_G^*(\text{pt})$-module isomorphism.
\end{theorem}

\section{A preliminary description of $KR^*_G(G^-)$}
\begin{lemma}\label{anticlassifying}
	Let $X$ be a finite $CW$-complex equipped with an involution. We have that
	\begin{align*}
		KR^1(X)&\cong[X, (U(\infty), a_\mathbb{R})]_\mathbb{R}\\
		KR^{-3}(X)&\cong[X, (U(2\infty), a_\mathbb{H})]_\mathbb{R}
	\end{align*}
	where $[X, Y]_\mathbb{R}$ means the set of Real homotopy equivalence classes of Real maps from $X$ to $Y$. Here Real homotopy equivalence is the one witnessed by a family of Real maps.
\end{lemma}
\begin{proof}
	The representing space for $KR^{-q}$ is a Real space which is homeomorphic to the representing space for $K^{-q}$ and whose subspace fixed under the Real structure is homeomorphic to that for $KO^{-q}$. By the definition of Real vector bundles and $KR$-theory, and noting that the representing spaces for $K$-theory and $KO$-theory are homogeneous spaces constructed out of various infinite dimensional matrix groups (orthogonal, unitary, and symplectic), the representing spaces for $KR$-theory are obtained from those for $KO$-theory by `complexifying' the relevant matrix groups (to be distinguished from the usual complexification of Lie groups). For example, $O(\infty)$ is transformed to $(U(\infty), \sigma_\mathbb{R})$ (analogous to the fact that $\mathbb{R}$ is complexified to $\mathbb{C}$), $\mathit{Sp}(\infty)$ to $(U(2\infty), \sigma_\mathbb{H})$ (analogous to the fact that $\mathbb{H}\otimes_\mathbb{R}\mathbb{C}\cong \text{M}_2(\mathbb{C})$), and $U(\infty)$ to $(U(\infty)\times U(\infty), (g_1, g_2)\mapsto (\overline{g}_2, \overline{g}_1))$ (analogous to the fact that $\mathbb{C}\otimes_\mathbb{R}\mathbb{C}\cong \mathbb{C}\oplus\mathbb{C}$). We shall show the case for $KR^{1}$. The other case is similar. The representing space for $KR^{1}$ is $(U(\infty)\times U(\infty)/U(\infty)_\Delta, (g_1, g_2)U(\infty)_\Delta\mapsto (\overline{g}_2, \overline{g}_1)U(\infty)_\Delta)$ (here $\Delta$ means diagonal subgroup), which is obtained from that for $KO^{1}$, which in turn is $U(\infty)/O(\infty)$ (cf. \cite{Bot}). We have that $(U(\infty), a_\mathbb{R})$ is Real diffeomorphic to $(U(\infty)\times U(\infty)/U(\infty)_\Delta, (g_1, g_2)U(\infty)_\Delta\mapsto (\overline{g_2}, \overline{g_1})U(\infty)_\Delta)$ by the map $g\mapsto (g, e)U(\infty)$. 
\end{proof}

\begin{definition}
	Let $\delta_\mathbb{R}^{\text{inv}}: RR(G)\to KR^{1}(G^-)$ send $\rho$ to the Real homotopy class of it, viewed as the Real map $G^-\to (U(\infty), a_\mathbb{R})$. Define $\delta_\mathbb{H}^{\text{inv}}: RH(G)\to KR^{-3}(G^-)$ similarly. 
\end{definition}

\begin{proposition}\label{cpxvecbdleinv}
	If $\rho$ be in $RR(G)$ with $(V, \sigma_V)$ being the underlying finite dimensional Real vector space of the Real unitary representation, then $\delta_\mathbb{R}(\rho)$ is represented by the following complex of Real vector bundles
	\begin{align*}
		0\longrightarrow G\times\mathbb{R}\times\mathbb{C}\times (V\oplus V)&\longrightarrow G\times\mathbb{R}\times \mathbb{C}\times(V\oplus V)\longrightarrow 0\\
		(g, t, z, v_1, v_2)&\mapsto\begin{cases}\left(g, t, z, \begin{pmatrix}-t\rho(g) &\overline{z}I_V\\ zI_V&t\rho(g)^*\end{pmatrix}\begin{pmatrix} iv_1\\iv_2\end{pmatrix}\right)&\ \ \text{ if }t\geq 0\\ \left(g,t, z, \begin{pmatrix}tI_V&\overline{z}I_V\\ zI_V& -tI_V\end{pmatrix}\begin{pmatrix}iv_1\\ iv_2\end{pmatrix}\right)&\ \ \text{ if }t\leq 0\end{cases}
	\end{align*}
	where the Real structure on $G\times\mathbb{R}\times\mathbb{C}\times(V\oplus V)$ is given by 
	\[(g, t, z, v_1, v_2)\mapsto (\sigma_G(g)^{-1}, t, -z, \sigma_V(v_2), \sigma_V(v_1))\]
	Similarly, if $\rho\in RH(G)$ with $(V, J_V)$ being the underlying finite dimensional Quaternionic vector space of the Quaternionic unitary representation, then $\delta_\mathbb{H}(\rho)$ is represented by the same complex of Real vector bundles except that the Real structure on $G\times\mathbb{R}\times\mathbb{C}\times(V\oplus V)$ is given by 
	\[(g, t, z, v_1, v_2)\mapsto (\sigma_G(g)^{-1}, t, z, -J_V(v_2), J_V(v_1))\]
\end{proposition}
\begin{proof}
	It is straightforward to verify that the given Real structures indeed commute with the middle maps of the complex of vector bundles, and that they are canonical. The complex of vector bundles, with the Real structures forgotten, is the tensor product of the following two complexes
	\begin{align*}
		0\longrightarrow G\times\mathbb{C}\times \mathbb{C}&\longrightarrow G\times\mathbb{C}\times \mathbb{C}\longrightarrow 0\\
		(g, z, w)&\mapsto (g, z, izw)
	\end{align*}
	\begin{align*}
		0\longrightarrow G\times\mathbb{R}\times V&\longrightarrow G\times\mathbb{R}\times V\longrightarrow 0\\
		(g, t, v)&\mapsto \begin{cases}(g, t, -it\rho(g)v)&\ \ \text{if }t\geq 0\\ (g, t, itv)&\ \ \text{if }t\leq 0\end{cases}
	\end{align*}
	(cf. \cite[Proposition 10.4]{ABS}) which represent the Bott class $\beta\in K^{-2}(G)$ and $\delta(\rho)\in K^{-1}(G)$ as defined in \cite{BZ} respectively (the middle maps of the above two complexes differ from the ones conventionally used to define $\beta$ and $\delta(\rho)$ by multiplication by $i$, which is homotopy equivalent to the identity map). Besides, the $KR$-theory classes represented by the complexes of Real vector bundles live in degree 1 and $-3$ pieces respectively because of the type of the involution of the middle maps restricted to $\mathbb{R}\times \mathbb{C}$. In sum, the two complexes of Real vector bundles represent canonical Real lifts of $\delta(\rho)$. Therefore they must represent $\delta_\mathbb{R}^\text{inv}(\rho)$ (resp. $\delta_\mathbb{H}^\text{inv}(\rho)$). 
\end{proof}
\begin{definition}
	Let $\delta_\mathbb{R}^{G, \text{inv}}: RR(G)\to KR_{G}^1(G^-)$ send $\rho$ to the complex of Real vector bundles as in Proposition \ref{cpxvecbdleinv} equipped with the equivariant structure given by 
	\[(g, t, z, v_1, v_2)\mapsto (\sigma_G(g)^{-1}, t, z, \rho(g)v_1, \rho(g)v_2)\]
	Define $\delta_\mathbb{H}^{G, \text{inv}}: RH(G)\to KR^{-3}_{G}(G^-)$ similarly.
\end{definition}
\begin{proposition}\label{antider}
	Identifying $RR(G)$ with $KR^0_{G}(\text{pt})$ and $RH(G)$ with $KR^{-4}_{G}(\text{pt})$ (cf. \cite[Sect. 8]{AS}), $\delta_\mathbb{R}^{G, \text{inv}}\oplus\delta_\mathbb{H}^{G, \text{inv}}$ is a derivation of the graded ring $KR^0_{G}(\text{pt})\oplus KR^{-4}_{G}(\text{pt})$ taking values in $KR^1_{G}(G^-)\oplus KR^{-3}_{G}(G^-)$.
\end{proposition}
\begin{proof}
	The proof can be easily adapted from the one of \cite[Proposition 3.1]{BZ} by straightforwardly modifying the homotopy $\rho_s$ and replacing the definition of the map $\delta_G$ given there (which is incorrect) with the one in \cite[Definition 2.5]{Fo}. The modified homotopy can be easily seen to intertwines with both Real structures of the complex of Real vector bundles as in Proposition \ref{cpxvecbdleinv}.
\end{proof}

\begin{proposition}
	$\delta_G(\overline{a_G^*}\rho)=-\delta_G(\overline{\sigma_G^*}\rho)$
\end{proposition}
\begin{proof}
	Viewing $\overline{\sigma_G^*}\rho$ and $\overline{a_G^*}\rho$ as maps from $G$ to $U(\infty)$, $\overline{\sigma_G^*}\rho\cdot \overline{a_G^*}\rho$ is the constant map with image being the identity. It follows that 
	\[0=\delta_G(\overline{\sigma_G^*}\rho\cdot\overline{a_G^*}\rho)=\delta_G(\overline{\sigma_G^*}\rho)+\delta_G(\overline{a_G^*}\rho)\]
	The last equality is the equivariant analogue of \cite[Lemma 2.4.6]{At}. 
\end{proof}

The fundamental representations of $G$ are permuted by $\overline{\sigma_G^*}$ (cf. \cite[Lemma 5.5]{Se}). Following the notations in \cite{Se} and \cite{Fo}, we let $\varphi_1, \cdots, \varphi_r, \theta_1, \cdots, \theta_s, \gamma_1, \cdots, \gamma_t, \overline{\sigma_G^*}\gamma_1, \cdots, \overline{\sigma_G^*}\gamma_t$ be the fundamental representations of $G$, where $\varphi_i\in RR(G, \mathbb{R})$, $\theta_j\in RH(G, \mathbb{R})$ and $\gamma_k\in R(G, \mathbb{C})$. 

\begin{definition}
	Let $\lambda_k^{\text{inv}}$ be the element in $KR^0(G^-)$ constructed as in the proof of \cite[Proposition 4.6]{Se} such that $c(\lambda_k^{\text{inv}})=\beta^3\delta(\gamma_k)\delta(\overline{a_G^*}\gamma_k)=-\beta^3\delta(\gamma_k)\delta(\overline{\sigma_G^*}\gamma_k)$. Let $\lambda_k^{G, \text{inv}}$ be the Real equivariant lift of $\lambda_k^{\text{inv}}$ in $KR^0_{G}(G^-)$ constructed by adding the natural equivariant structure throughout the construction of $\lambda_k^{\text{inv}}$ such that $c(\lambda_k^{G, \text{inv}})=\beta^3\delta_G(\gamma_k)\delta_G(\overline{a_G^*}\gamma_k)=-\beta^3\delta_G(\gamma_k)\delta_G(\overline{\sigma_G^*}\gamma_k)$. 
\end{definition}	
Applying \cite[Theorem 4.2]{Se}, one can get the $KR^*(\text{pt})$-module structure of $KR^*(G^-)$, which is generated, as a $KR^*(\text{pt})$-algebra, by $\delta_\mathbb{R}^\text{inv}(\varphi_1), \cdots, \delta_\mathbb{R}^\text{inv}(\varphi_r)$, $\delta_\mathbb{H}^\text{inv}(\theta_1), \cdots, \delta_\mathbb{H}^\text{inv}(\theta_s)$, $\lambda_1^\text{inv}, \cdots, \lambda_t^\text{inv}$, and realifications of certain products of $\beta^i$ and $\delta(\gamma_k)$, $k=1, \cdots, t$ (compare with \cite[Theorem 5.6]{Se}). Noting that all these generators admit Real equivariant lifts, and that $G$ is a weakly equivariantly formal space by Remark \ref{gequivformal}, we have that $G^-$ is a Real equivariantly formal space by Definition \ref{equivformal} (\ref{realequivformaldef}). Now Theorem \ref{equivstrthm} applies and one can further obtain the $KR_{G}^*(\text{pt})$-module structure of $KR^*_{G}(G^-)$. We shall state the following description of $KR^*_{G}(G^-)$ without proof. We refer the reader to \cite[Corollaries 4.10, 4.11, Proposition 4.13, Theorem 4.33]{Fo} for comparison.

\begin{theorem}\label{antimodstr}
	\begin{enumerate}
		\item The map 
		\begin{align*}
			f: (RR(G, \mathbb{R})\oplus RH(G, \mathbb{R}))\otimes KR^*(G^-)\oplus r(R(G, \mathbb{C})\otimes K^*(G))&\to KR^*_{G}(G^-)\\
			\rho_1\otimes x_1\oplus r(\rho_2\otimes x_2)&\mapsto \rho_1\cdot (x_1)_G\oplus r(\rho_2\cdot (x_2)_G)
		\end{align*}
		is a group isomorphism, where $x_G\in KR^*_{G}(G^-)$ is a Real equivariant lift of $x\in KR^*(G^-)$. If $R(G, \mathbb{C})=0$, then $f$ is an isomorphism of $KR^*_{G}(\text{pt})$-modules.
		\item $KR^*_{G}(G^-)$ is generated as an algebra over $KR_{G}^*(\text{pt})$ (for descriptions of the coefficient ring see \cite[Section 3]{Fo}) by $\delta_\mathbb{R}^{G, \text{inv}}(\varphi_1), \cdots, \delta_\mathbb{R}^{G, \text{inv}}(\varphi_r)$, $\delta_\mathbb{H}^{G, \text{inv}}(\theta_1)$, $\cdots, \delta_\mathbb{H}^{G, \text{inv}}(\theta_s)$, $\lambda_1^{G, \text{inv}}, \cdots, \lambda_t^{G, \text{inv}}$ and 
	\[\{r_{\rho, i, \varepsilon_1, \cdots, \varepsilon_t, \nu_1, \cdots, \nu_t}^{G, \text{inv}}:= r(\beta^i\cdot\rho\delta_G(\gamma_1)^{\varepsilon_1}\cdots\delta_G(\gamma_t)^{\varepsilon_t}\delta_G(\overline{a_G^*}\gamma_1)^{\nu_1}\cdots\delta_G(\overline{a_G^*}\gamma_t)^{\nu_t}\}\]
	where $\rho\in R(G, \mathbb{C})\oplus\mathbb{Z}\cdot\rho_{\text{triv}}$, $\varepsilon_1, \cdots, \varepsilon_t$, $\nu_1, \cdots, \nu_t$ are either 0 or 1, $\varepsilon_k$ and $\nu_k$ are not equal to 1 at the same time for $1\leq k\leq t$, and the first index $k_0$ where $\varepsilon_{k_0}=1$ is less than the first index $k_1$ where $\nu_{k_1}=1$. Moreover,
		\begin{enumerate}
			\item $(\lambda_k^{G, \text{inv}})^2=0$.
			\item Let $\omega_t:=\delta_{\varepsilon_t, 1-\nu_t}$. Then 
			\[(r^G_{\rho, i, \varepsilon_1, \cdots, \varepsilon_t, \nu_1, \cdots, \nu_t})^2=\begin{cases}\eta^2(\rho\cdot\overline{\sigma_G^*}\rho)(\lambda^{G, \text{inv}}_1)^{\omega_1}\cdots(\lambda_t^{G, \text{inv}})^{\omega_t}&\ \ \text{if }r^{G, \text{inv}}_{\rho, i, \varepsilon_1, \cdots, \varepsilon_t, \nu_1, \cdots, \nu_t}\\
				&\ \ \text{is of degree }-1\text{ or }-5\\ 
				\pm\mu(\rho\cdot\overline{\sigma_G^*}\rho)(\lambda^{G, \text{inv}}_1)^{\omega_1}\cdots(\lambda_t^{G, \text{inv}})^{\omega_t}&\ \ \text{if }r^{G, \text{inv}}_{\rho, i, \varepsilon_1, \cdots, \varepsilon_t, \nu_1, \cdots, \nu_t}\\
				&\ \ \text{is of degree }-2\text{ or }-6\\
				0&\ \ \text{otherwise}
			      	\end{cases}\]
				The sign can be determined using formulae in \cite[Proposition 2.29 (2)]{Fo}.	
			\item $r^{G, \text{inv}}_{\rho, i, \varepsilon_1, \cdots, \varepsilon_t, \nu_1, \cdots, \nu_t}\eta=0$, and $r^{G, \text{inv}}_{\rho, i, \varepsilon_1, \cdots, \varepsilon_t, \nu_1, \cdots, \nu_t}\mu=2r^{G, \text{inv}}_{\rho, i+2, \varepsilon_1, \cdots, \varepsilon_t, \nu_1, \cdots, \nu_t}$.	
		\end{enumerate}
	\end{enumerate}
\end{theorem}

\begin{corollary}\label{antimodstrnocpx}
	In particular, if $R(G, \mathbb{C})=0$, then 
	\begin{align*}
		KR^*_{G}(G^-)&=\bigwedge\nolimits_{KR^*_{G}(\text{pt})}(\delta_\mathbb{R}^{G, \text{inv}}(\varphi_1), \cdots, \delta_\mathbb{R}^{G, \text{inv}}(\varphi_r), \delta_\mathbb{H}^{G, \text{inv}}(\theta_1), \cdots, \delta_\mathbb{H}^{G, \text{inv}}(\theta_s))\\
		                                                       &\cong \Omega_{KR^*_{G}(\text{pt})/KR^*(\text{pt})}
	\end{align*}
	as $KR^*_{G}(\text{pt})$-modules.
\end{corollary}
As we can see, the module structure of $KR^*_{G}(G^-)$ is very similar to that of $KR^*_{G}(G)$, except that the degrees of the generators are different. Now it remains to find $\delta_\mathbb{R}^{G, \text{inv}}(\varphi_i)^2$ and $\delta_\mathbb{H}^{G, \text{inv}}(\theta_j)^2$ so as to complete the description of the ring structure of $KR^*_{G}(G^-)$. As it turns out, these squares are all zero, in stark contrast to the involutive automorphism case.

\section{Squares of the real and quaternionic type generators}
This section is devoted to proving that the squares of the real and quaternionic generators are zero, following the strategy outlined in \cite[Section 4]{Fo}. We use $T$ to denote the maximal torus of $U(n)$ consisting of diagonal matrices throughout this section.

Applying Brylinski-Zhang's result on the equivariant $K$-theory of compact connected Lie group $G$ with $\pi_1(G)$ torsion-free and Theorem \ref{equivstrthm}, we have
\begin{proposition}
For $\mathbb{F}=\mathbb{R}$ or $\mathbb{H}$, we have the following $KR_{(U(n), \sigma_\mathbb{F})}^*(\text{pt})$-module isomorphism
\[KR^*_{(U(n), \sigma_\mathbb{F})}(U(n), a_\mathbb{F})\cong \Omega_{KR^*_{(U(n), \sigma_\mathbb{F})}(\text{pt})/KR^*(\text{pt})}\]
The set $\{\delta_\mathbb{R}^{G, \text{inv}}(\sigma_n), \delta_\mathbb{R}^{G, \text{inv}}(\bigwedge^2\sigma_n), \cdots, \delta_\mathbb{R}^{G, \text{inv}}(\bigwedge^n\sigma_n)\}$ is a set of primitive generators for the case $\mathbb{F}=\mathbb{R}$, while $\{\delta_\mathbb{H}^{G, \text{inv}}(\sigma_{2m}), \delta_\mathbb{R}^{G, \text{inv}}(\bigwedge^2\sigma_{2m}), \cdots, \delta_\mathbb{R}^{G, \text{inv}}(\bigwedge^{2m}\sigma_{2m})\}$ is a set of primitive generators for the case $\mathbb{F}=\mathbb{H}$. 
\end{proposition}

\begin{corollary}\label{antiformality}
	We have the following isomorphism
	\[KR^*_{(U(n), \sigma_\mathbb{F})}(U(n), a_\mathbb{F})\cong \Omega_{R(U(n))/\mathbb{Z}}\otimes KR^*(\text{pt})\]
	as ungraded $KR^*(\text{pt})$-modules.
\end{corollary}

\begin{definition}
	Let 
	\[p_{G, \text{inv}}^*: KR^*_{(U(n), \sigma_\mathbb{R})}(U(n), a_\mathbb{R})\to KR^*_{(T, \sigma_\mathbb{R})}(T, \text{Id})\]
	be the restriction map and the map
	\[q_{G, \text{inv}}^*: KR^*_{(U(2m), \sigma_\mathbb{H})}(U(2m), a_\mathbb{H})\to KR^*_{(U(2m), \sigma_\mathbb{H})}(U(2m)/T\times T, \sigma_\mathbb{H}\times\text{Id})\]
	induced by the Weyl covering map 
	\begin{align*}
		q_G: U(2m)/T\times T&\to U(2m)\\
		(gT, t)&\mapsto gtg^{-1}
	\end{align*}
\end{definition}

\begin{proposition}\label{antiatiyahmap}
	Identifying $KR^*_{(T, \sigma_\mathbb{R})}(T, \text{Id})$ with $RR(T, \sigma_\mathbb{R})\otimes KR^*(T, \text{Id})$, we have
	\[p_{G, \text{inv}}^*(\delta_\mathbb{R}^{G, \text{inv}}(\bigwedge\nolimits^k\sigma_n))=\sum_{1\leq j_1<\cdots< j_k\leq n}e_{j_1}\cdots e_{j_k}\otimes\delta_\mathbb{R}^{\text{inv}}(e_{j_1}+\cdots+e_{j_k})\]
	where $e_i$ is the 1-dimensional Real representation of $(T, \sigma_\mathbb{R})$ with weight being the $i$-th standard basis vector of the weight lattice. Similarly, identifying $KR_{(U(2m), \sigma_\mathbb{H})}(U(2m)/T\times T, \sigma_\mathbb{H}\times\text{Id})$ with $\mathbb{Z}[e_1^{\mathbb{H}}, \cdots, e_{2m}^{\mathbb{H}}, (e_1^\mathbb{H}\cdots e_{2m}^\mathbb{H})^{-1}]\otimes KR^*(T, \text{Id})$ (cf. \cite[Proposition 4.25]{Fo}), where $e_i^\mathbb{H}$ is the degree $-4$ class in $KR_{(U(2m), \sigma_\mathbb{H})}(U(2m)/T, \sigma_\mathbb{H})$ represented by the Quaternionic line bundle $U(2m)\times_T\mathbb{C}_{e_i}$, we have, for $\mathbb{F}=\mathbb{R}$ or $\mathbb{H}$ (depending on the parity of $k$), 
	\[q_{G, \text{inv}}^*(\delta_\mathbb{F}^{G, \text{inv}}(\bigwedge\nolimits^k\sigma_{2m}))=\sum_{1\leq j_1<\cdots<j_k\leq 2m}e_{j_1}^\mathbb{H}\cdots e_{j_k}^\mathbb{H}\otimes\delta_\mathbb{R}^\text{inv}(e_{j_1}+\cdots+e_{j_k})\]
\end{proposition}
\begin{proof}
	The proof is similar to \cite[Lemma 4.19]{Fo}. The Proposition follows from the fact that the complex of $U(n)$-equivariant Real vector bundles representing $\delta_\mathbb{F}^{G, \text{inv}}(\bigwedge\nolimits^k\sigma_n)$, as in Proposition \ref{cpxvecbdleinv}, is decomposed into a direct sum of complexes of $T$-equivariant Real vector bundles, each of which corresponds to a weight of $\bigwedge^k\sigma_n$.
\end{proof}

\begin{proposition}\label{antiinj}
	Both $p_{G, \text{inv}}^*$ and $q_{G, \text{inv}}^*$ are injective. 
\end{proposition}
\begin{proof}
	By \cite[Lemma 4.19 and Proposition 4.25]{Fo}, and Corollary \ref{antiformality} and Proposition \ref{antiatiyahmap}, we can identify both $p_{G, \text{inv}}^*$ and $q_{G, \text{inv}}^*$ with the map
	\[i^*\otimes \text{Id}_{KR^*(\text{pt})}: K^*_{U(n)}(U(n))\otimes KR^*(\text{pt})\to K_T^*(T)\otimes KR^*(\text{pt})\]
	where the restriction map $i^*$ can factor through $K_T^*(U(n))$ as 
	\[K_{U(n)}^*(U(n))\stackrel{i_1^*}{\longrightarrow} K_T^*(U(n))\stackrel{i_2^*}{\longrightarrow} K_T^*(T)\]
	$i_1^*\otimes \text{Id}_{KR^*(\text{pt})}$ is injective because $i_1^*$ is split injective by \cite[Proposition 4.9]{At3}. By adapting \cite[Lemma 4.20]{Fo} to the case $G=U(n)$, we have
	\[i_2^*\left(\prod_{i=1}^n\delta_T(\bigwedge\nolimits^i\sigma_n)\right)=d_{U(n)}\otimes\prod_{i=1}^n\delta(e_i)\]
	where $d_{U(n)}$ is the Weyl denominator for $U(n)$. By \cite[Lemma 4.21]{Fo} and the fact that $rd_{U(n)}\otimes\prod_{i=1}^n\delta(e_i)\neq 0$ for all $r\in KR^*(\text{pt})\setminus\{0\}$, $i_2^*\otimes\text{Id}_{KR^*(\text{pt})}$ is injective as well. Thus $i^*\otimes \text{Id}_{KR^*(\text{pt})}$, as well as $p_{G, \text{inv}}^*$ and $q_{G, \text{inv}}^*$, are injective. 
\end{proof}
\begin{lemma}\label{circlesqzero}
	Let $e$ be the standard representation of $S^1$. Then $\delta_\mathbb{R}^\text{inv}(e)^2=0$ in $KR^*(S^1, \text{Id})$.
\end{lemma}
\begin{proof}
	Note that $\delta_\mathbb{R}^\text{inv}(e)\in KR^{-7}(S^1, \text{Id})$. So $\delta_\mathbb{R}^\text{inv}(e)^2\in KR^{-6}(S^1, \text{Id})\cong KR^{-7}(\text{pt})=0$. 
\end{proof}
\begin{proposition}\label{sqzerounitary}
	For $\mathbb{F}=\mathbb{R}$ or $\mathbb{H}$, $\delta_\mathbb{F}^{G, \text{inv}}(\sigma_n)^2=0$ in $KR_{(U(n), \sigma_\mathbb{F})}^*(U(n), a_\mathbb{F})$. 
\end{proposition}
\begin{proof}
	This follows from Propositions \ref{antiatiyahmap} and \ref{antiinj} and Lemma \ref{circlesqzero}. 
\end{proof}
The above results finally culminate in the main theorem of this note. 
\begin{theorem}\label{sqzeroantibz}
	\begin{enumerate}
		\item\label{sqzero} Let $G$ be a Real compact Lie group, and $\rho$ a Real (resp. Quaternionic) unitary representation of $G$. Then $\delta_\mathbb{F}^{G, \text{inv}}(\rho)^2=0$ in $KR_{G}^*(G^-)$ for $\mathbb{F}=\mathbb{R}$ (resp. $\mathbb{F}=\mathbb{H}$).
		\item\label{antibz} In particular, if $G$ is connected and simply-connected and $R(G, \mathbb{C})=0$, then $\delta_\mathbb{R}^{G, \text{inv}}\oplus \delta_\mathbb{H}^{G, \text{inv}}$ induces the following ring isomorphism
			\[KR^*_{G}(G^-)\cong\Omega_{KR^*_{G}(\text{pt})/KR^*(\text{pt})}\]
	\end{enumerate}
\end{theorem}
\begin{proof}
	Note that the induced map $\rho^*: KR^*_{(U(n), \sigma_\mathbb{F})}(U(n), a_\mathbb{F})\to KR_{G}^*(G^-)$ sends $\delta_\mathbb{F}^{G, \text{inv}}(\sigma_n)$ to $\delta_\mathbb{F}^{G, \text{inv}}(\rho)$ by the interpretation of $\delta_\mathbb{F}^{G, \text{inv}}$ in Proposition \ref{cpxvecbdleinv}. Now part (\ref{sqzero}) follows from Proposition \ref{sqzerounitary}. Part (\ref{antibz}) follows from part (\ref{sqzero}), Proposition \ref{antider} and Corollary \ref{antimodstrnocpx}.
\end{proof}
Note that Theorems \ref{antimodstr} and \ref{sqzeroantibz} give a complete description of the ring structure of $KR_{G}^*(G^-)$. Part (\ref{antibz}) of Theorem \ref{sqzeroantibz} should be viewed as a generalization of Brylinski-Zhang's result in the context of $KR$-theory. 

Last but not least, we obtain, as a by-product, the following
\begin{corollary}\label{xsquare}
	If $G$ is a compact connected Real Lie group (not necessarily simply-connected) and $X$ a compact Real $G$-space, then for any $x$ in $KR^1_G(X)$ or $KR_G^{-3}(X)$, $x^2=0$. 
\end{corollary}
\begin{proof}
	Let $EG^n$ be the join of $n$ copies of $G$, with the Real structure induced by $\sigma_G$ and $G$-action by the left-translation of $G$. Let $\pi_n^*: KR^*_G(X)\to KR_G^*(X\times EG^n)$ be the map induced by projection onto $X$. Consider the map
	\[\pi^*:=\lim_{\stackrel{\longleftarrow}{n}} \pi_n^*: KR_G^*(X)\to \lim_{\stackrel{\longleftarrow}{n}} KR_G^*(X\times EG^n)\]
	By adapting the proof of \cite[Corollary 2.3]{AS} to the Real case, $\text{ker}(\pi^*)=\bigcap_{n\in\mathbb{N}} IR^n\cdot KR_G^*(X)$, where $IR$ is the augmentation ideal of $RR(G)$. Moreover, $\bigcap_{n\in\mathbb{N}} I^n=(0)$ where $I$ is the augmentation ideal of $R(G)$ (cf. the Note immediately preceding \cite[Section 4.5]{AH}), and the map $RR(G)\to R(G)$ forgetting the Real structure is injective (cf. \cite[Proposition 2.17(1)]{Fo}). It follows that $\bigcap_{n\in\mathbb{N}}IR^n=(0)$ and thus $\pi^*$ is injective. Now it suffices to show that $\pi^*(x)^2=0$. Using Lemma \ref{anticlassifying} and compactness of $X\times EG^n/G$, $\pi^*_n(x)\in KR^*_G(X\times EG^n)=KR^*(X\times EG^n/G)$ can be represented by a Real map $f_n: X\times EG^n/G\to (U(k_n), a_\mathbb{F})$ for some $k_n$. So $\pi_n^*(x)=f_n^*\delta_\mathbb{F}^{\text{inv}}(\sigma_{k_n})$, and $\pi^*_n(x)^2=f_n^*\delta_\mathbb{F}^{\text{inv}}(\sigma_{k_n})^2=0$ because we have that $\delta_{\mathbb{F}}^{\text{inv}}(\sigma_{k_n})^2=0$ by applying the forgetful map $KR_{(U(n), \sigma_\mathbb{F})}^*(U(n), a_\mathbb{F})\to KR^*(U(n), a_\mathbb{F})$ to the equation $\delta_\mathbb{F}^{G, \text{inv}}(\sigma_{k_n})^2=0$ (see Proposition \ref{sqzerounitary}). Since $\pi^*(x)^2$ is the inverse limit of $\pi_n^*(x)^2=0$, $\pi^*(x)^2=0$ as desired. 
\end{proof}
\begin{remark}
	Corollary \ref{xsquare} has the following consequences as a result of forgetting some structures in the equivariant $KR$-theory. 
	\begin{enumerate}
		\item When $G$ is trivial, Corollary \ref{xsquare} then says that $x^2=0$ for any $x$ in $KR^i(X)$, where $i$ is 1 or $-3$. This special case can be proved more directly using Lemma \ref{anticlassifying}, and noting that $\delta_\mathbb{F}^{\text{inv}}(\sigma_n)^2=0$ in $KR^*(U(n), a_\mathbb{F})$, which can be obtained by applying the map $KR_{U(n), \sigma_\mathbb{F}}^*(U(n), a_\mathbb{F})\to KR^*(U(n), a_\mathbb{F})$ forgetting the equivariant structure to the equation in Proposition \ref{sqzerounitary}. 
		\item Applying the map $KR^*(U(n), a_\mathbb{F})\to K^*(U(n))$, which further forgets the Real structure, to the equation $\delta_\mathbb{F}^{\text{inv}}(\sigma_n)^2=0$, yields $\delta(\sigma_n)^2=0$ in $K^*(U(n))$. Pulling back from this `universal' equation, we then get the well-known fact that $x^2=0$ for $x\in K^{-1}(X)$. 
		\item If the involution on $X$ is trivial, then $KR^*(X)\cong KO^*(X)$, and we also have that $x^2=0$ for $x\in KO^i(X)$, where $i$ is 1 or $-3$. 
	\end{enumerate}
	However, it is not true that $x^2=0$ if $x\in KR_G^i(X)$ for $i=-1$ or $-5$. For example, if $\eta$ is the non-zero element in $KR^{-1}(\text{pt})$, then $\eta^2\neq 0$. An example where $x\in KR_G^{-5}(X)$ but $x^2\neq 0$ can be found in \cite[Corollary 4.32]{Fo}.
\end{remark}

\footnotesize\textsc{National Center for Theoretical Sciences,\\
Mathematics Division,\\
National Tsing Hua University\\
Hsinchu 30013, Taiwan\\E-mail: }\texttt{ckfok@ntu.edu.tw}\\
\textsc{URL: }\texttt{http://www.math.cornell.edu/$\sim$ckfok}
\end{document}